\documentclass[12pt]{amsart}
\usepackage{lipsum}
\usepackage{amssymb}
\usepackage{enumerate}
\usepackage{amsthm}
\usepackage{tikz-cd}
\usepackage{amsmath}
\usepackage[mathscr]{euscript}
\usepackage[utf8]{inputenc}
\usepackage[english]{babel}
\usepackage{hyperref}
\hypersetup{
    colorlinks=true,
    linkcolor=blue,
    citecolor=red,
    filecolor=red,      
    urlcolor=violet,
}

\makeatletter
\@namedef{subjclassname@2020}{%
  \textup{2020} Mathematics Subject Classification}
\makeatother
\newtheorem{thm}{Theorem}[section]

\newtheorem{defi}[thm]{Definition}

\newtheorem{conj}[thm]{Conjecture}
\newtheorem{qn}[thm]{Question}

\DeclareMathOperator{\rk}{{rk}}

\begin{document}
\baselineskip=16pt

\subjclass[2020]{Primary 14J60}
\keywords{}

\author{Snehajit Misra}
\author{Nabanita Ray}

\address{Harish Chandra Research Institute,Chhatnag Road, Jhunsi
Prayagraj (Allahabad) 211 019
India.}
\email[Snehajit Misra]{misra08@gmail.com}
\address{Indraprastha Institute of Information Technology, Delhi (IIITD),
Okhla Phase-III, New Delhi-110020, India.}
\email[Nabanita Ray]{nabanitaray2910@gmail.com}

\begin{abstract}
 In this article, we investigate the stability of syzygy bundles corresponding to ample and globally generated vector bundles on smooth irreducible projective surfaces.
\end{abstract}

\title{On Stability of Syzygy Bundles}

\maketitle

\section{Introduction}
Let $X$ be a smooth irreducible projective variety of dimension $n$ over an algebraically closed field $k$, and let $E$ be a globally generated vector bundle of rank $r$ on $X$.
The syzygy bundle associated to the vector bundle $E$, denoted by $M_E$ is defined as the kernel of the evaluation map $$ev : H^0(X,E)\otimes \mathcal{O}_X\longrightarrow E.$$ Thus we have the following short exact sequence:
\begin{center}
 $0\longrightarrow M_{E}\longrightarrow H^0(X,E)\otimes \mathcal{O}_X\longrightarrow E\longrightarrow 0.$
\end{center}
These vector bundles  arise in a variety of geometric and algebraic
problems. For example what are the numerical conditions  on vector bundles $E$
and $F$ on $X$ such that the natural product map
$$H^0(X,E)\otimes H^0(X,F) \longrightarrow H^0(X,E\otimes F)$$ becomes surjective.
The stability of $M_E$ plays a crucial role in the understanding of the surjectivity of the above question. For example, the stability of $M_L$ is studied in \cite{PR87} and \cite{C08} when $L$ is a globally generated line bundle on a smooth complex irreducible projective curve $C$. More generally, the stability of these bundles $M_E$ have been studied in \cite{B94} when $E$ is a vector bundle on a smooth curve $C$. In particular, it is shown that if $E$ is semistable (see Definition \ref{defi3.1} for semistability) globally generated vector bundle and $\mu(E)\geq 2g(C)$, where $g(C)$ is the genus of the curve $C$, then $M_E$ is semistable. Furthermore, if  $E$ is stable globally generated vector bundle on a smooth curve $C$ under the hypothesis $\mu(E)\geq 2g(C)$,  then  $M_E$ is stable unless $\mu(E) = 2g(C)$ and either $C$ is hyperelliptic or $\Omega_C\hookrightarrow E$.

In \cite{ELM13}, the authors
study the stability of $M_L$  when $L$ is a very ample line bundle on a smooth surface. In fact they showed the following:  if we fix an ample divisor $H$ and an arbitrary divisor $P$ on an algebraic surface $X$ and given a large integer $d$, set
$$L_d = dH + P$$
and write $M_d = M_{L_d}$, then for sufficiently large $d$, the syzygy bundle $M_d$ is slope stable with respect
to ample divisor $H$. Ein-Lazarsfield-Mustopa conjectured the following about stability of  syzygy bundles corresponding to line bundles on higher dimesional varieties in \cite{ELM13}. \begin{conj}\label{conj1.1}
\cite[Conjecture 2.6]{ELM13} Let $X$ be a smooth projective variety of dimension $n$, and define $M_d$ as above. Then $M_d$ is $H$-stable for every $d\gg 0$.
\end{conj}
This above mentioned Conjecture \ref{conj1.1} is solved in \cite{R24} (see \cite[Theorem 4.3, Corollary 4.4]{R24}.)

However, the results in \cite{ELM13} is not an effective one. An effective version about stability of $M_L$ is given in \cite{C11},\cite{C12},\cite{LZ22},\cite{R24} for ample and globally generated line bundles $L$ mostly over various smooth projective surfaces.
\vspace{2mm}

Recently, the stability of $M_E$ is studied for stable vector bundles $E$ of higher rank on a smooth projective surface $X$.
For any vector
bundle $E$ on a smooth polarized variety $(X,H)$ and $m > 0$, we denote $E(m) := E \otimes \mathcal{O}_X(mH)$. The authors in \cite{BP23} prove that if $E$ is an $H$-stable vector bundle on a smooth projective surface $X$, then $M_{E(m)}$ is $H$-stable for $m\gg 0$.
However no such effective result is known for stability of syzygy bundles corresponding to stable globally generated vector bundles on smooth surfaces.
\vspace{2mm}

The purpose of this paper is to investigate the stability of syzygy bundles associated to
a globally generated vector bundles on a smooth projective algebraic surface. In this paper we prove the following :

\begin{thm}\label{thm1}
 Let $X$ be a smooth irreducible complex projective surface. Let $E$ be an ample and globally generated vector bundle on $X$, $H$ be an ample divisor on $X$ such that there exists an irreducible smooth curve $C\in  \lvert H\rvert$. Further assume that 
\begin{enumerate}
  \item $E\otimes \mathcal{O}_X(-H)$ is globally generated and $\dim_{\mathbb{C}}H^1\bigl(X,E\otimes \mathcal{O}_X(-H)\bigr) = 0.$
  \item If $W \subseteq E\otimes \mathcal{O}_X(-H)$ generates $E\otimes \mathcal{O}_X(-H)$, then the natural multiplication map
$$H^0(X,\mathcal{O}_X(H)\otimes m_x)\bigr) \otimes  W \longrightarrow H^0(X,E\otimes m_x\bigr)$$
is surjective.
  \item $\dim_{\mathbb{C}}H^0\bigl(X,\mathcal{O}_X(H)\bigr)\geq \dim_{\mathbb{C}}H^0\bigl(C,E\vert_C\bigr)-\rk(E)+1$.
  \item $M_{E\vert_C}$ is semistable.
 \end{enumerate}
Then $M_E$ is $H$-stable.
 \end{thm}
 As a corollary to the above theorem, we have the following result extending the result in \cite{BP23} to semistable globally generated vector bundles on smooth surfaces. Our proofs  are inspired by the ideas in \cite{ELM13}, \cite{LZ22} and \cite{BP23} respectively.
 \begin{thm}
  Let $X$ be a smooth complex surface and we fix a very ample line bundle $H$ on $X$. Let $E$ be an $H$-semistable globally generated vector bundle of rank $r$ on $X$.  Then for large enough $m\gg 0$, the associated sygyzy bundle $M_{E(m)}$ is $H$-stable.
 \end{thm}

\section{Notations and Conventions}
We work over the field of complex numbers $\mathbb{C}$. Given a coherent
sheaf $\mathcal{G}$ on a variety $X$, we write $h^i(\mathcal{G})$ to denote the dimension of the $i$-th cohomology group $H^i(X,\mathcal{G})$. The sheaf $K_X$ will denote the canonical sheaf on $X$ . Throught this article, (semi)stability means slope (semi)stability. The first Chern class of a vector bundle $E$ will be denoted by $c_1(E)$. The rank of a vector bundle $E$ will be denoted by rk$(E)$. For any closed
point $x$, let $m_x$ denotes the ideal defining the  point $x$. For any vector space $V$, $\mathbb{P}_{sub}(V)$ denotes the projective space associated to $V$ consisting of 1-dimensional linear subspace of $V$.
\section{Main Result}
The aim of this section is to prove Theorem \ref{thm1}.
Let $X$ be a smooth complex projective variety of dimension $n$ with a fixed ample line bundle $H$ on it.
For a non-zero vector bundle $V$ of rank $r$ on $X$, the $H$-slope of $V$ is defined as
\begin{align*}
\mu_H(V) := \frac{c_1(V)\cdot H^{n-1}}{r} \in \mathbb{Q},
\end{align*}
where $c_1(V)$ denotes the first Chern class of $V$.
\begin{defi}\label{defi3.1}
A vector bundle $V$ on $X$ is said to be $H$-semistable (respectively stable) if $\mu_H(W) \leq \mu_H(V)$ (respectively $\mu_H(W) < \mu_H(V)$) for all subsheaves $W$ of $V$. A vector bundle $V$ on $X$ is called $H$-unstable if it is not $H$-semistable.
\end{defi}
See \cite{HL10} for more details about semistability.
\vspace{1mm}

Recall the short exact sequence for a globally generated vector bundle $E$.
\begin{center}
 $0\longrightarrow M_{E}\longrightarrow H^0(X,E)\otimes \mathcal{O}_X\longrightarrow E\longrightarrow 0.$
\end{center}

Hence, the $H$-slope of a syzygy bundle $M_E$ associated to a globally generated vector bundle $E$ is as follows :
$$\mu_H(M_E) = \frac{c_1(M_E)\cdot H^{n-1}}{\rk(M_E)} = \frac{-c_1(E)\cdot H^{n-1}}{h^0(E)-\rk(E)}.
$$

In addition, if $E$ is ample, then $c_1(E)$ is also ample, and hence $\mu_H(M_E) < 0$ in this case.

\begin{thm}
 Let $X$ be a smooth irreducible complex projective surface. Let $E$ be an ample and globally generated vector bundle on $X$, $H$ be a very ample divisor on $X$ such that there exists an irreducible smooth curve $C\in  \lvert H\rvert$. Further assume that
 \begin{enumerate}
  \item $E\otimes \mathcal{O}_X(-H)$ is globally generated and $\dim_{\mathbb{C}}H^1\bigl(X,E\otimes \mathcal{O}_X(-H)\bigr) = 0.$
  \item If $W \subseteq E\otimes \mathcal{O}_X(-H)$ generates $E\otimes \mathcal{O}_X(-H)$, then the natural multiplication map
$$H^0(X,\mathcal{O}_X(H)\otimes m_x)\bigr) \otimes  W \longrightarrow H^0(X,E\otimes m_x\bigr)$$
is surjective.
  \item $\dim_{\mathbb{C}}H^0\bigl(X,\mathcal{O}_X(H)\bigr)\geq \dim_{\mathbb{C}}H^0\bigl(C,E\vert_C\bigr)-\rk(E)+1$.
  \item $M_{E\vert_C}$ is semistable.
 \end{enumerate}
Then $M_E$ is $H$-stable.
 \end{thm}
 \begin{proof}
  Assume that $M_E$ is $H$-unstable, and let $F\subset M_E$ be the maximal destabilizing subbundle such that $$\mu_H(F)\geq \mu_H(M_E).$$
  
  Since $C\in\lvert H\rvert$, we have
  \begin{align}\label{seq1}
  \mu(F\vert_C) \geq \mu\bigl((M_E)\vert_C\bigr).
  \end{align}
  As $H^1\bigl(X, E\otimes \mathcal{O}_X(-H)\bigr) = 0$, we have the following diagram:
\begin{center}
\begin{tikzcd}
0\arrow[r," "] & K \arrow[r, " "] \arrow[d, " "] & F\vert_C \arrow[r, " "] \arrow[d,""] & N \arrow[r," "] \arrow[d," "] & 0 \\
0\arrow[r," "] & H^0\bigl(X,E\otimes\mathcal{O}_X(-H)\bigr)\otimes \mathcal{O}_C \arrow[r, " "]          & (M_E)\vert_C\arrow[r, " "]                                         & M_{E\vert_C} \arrow[r," "] & 0
\end{tikzcd}
\end{center}
If $K=0$, then $\mu(N) = \mu(F\vert_C) \geq \mu\bigl((M_E)\vert_C\bigr).$ As $E$ is ample and globally generated, we have $E\vert_C$ is also ample and globally generated. Thus $\mu\bigl(M_{E\vert_C}\bigr) < 0$ by our observation. Now from the following exact sequence
\begin{align}\label{seq2}
0\longrightarrow H^0(X,E\otimes \mathcal{O}_X(-H))\otimes \mathcal{O}_C \longrightarrow (M_E)\vert_C \longrightarrow M_{E\vert_C}\longrightarrow 0
\end{align}
we get
\begin{align}\label{seq3}
\deg\bigl((M_E)\vert_C\bigr) = \deg(M_{E\vert_C}) < 0 \hspace{1mm} \text{and} \hspace{1mm} \rk\bigl((M_E)\vert_C\bigr) = h^0(X,E\otimes \mathcal{O}_X(-H)) + \rk(M_{E\vert_C})
\end{align}
so that $\mu\bigl((M_E)\vert_C\bigr) >\mu(M_{E\vert_C}) $.
Therefore we have
$$\mu(N) = \mu(F\vert_C) \geq \mu\bigl((M_E)\vert_C\bigr) > \mu(M_{E\vert_C}),$$
which gives a contradiction to the semistability of $M_{E\vert_C}$ (see hypothesis (4)). Thus we conclude $K\neq 0$. By hypothesis (2), the natural multiplication map of sections $$\nu : \mathbb{P}_{sub}\bigl(H^0(X,\mathcal{O}_X(H)\otimes m_x)\bigr) \times \mathbb{P}_{sub}\bigl(H^0(X,E\otimes \mathcal{O}_X(-H)\bigr) \longrightarrow \mathbb{P}_{sub}\bigl(H^0(X,E\otimes m_x)\bigr)$$
$$ \Bigl([s_1],[s_2]\Bigr) \longrightarrow [s_1\otimes s_2]$$
is a finite morphism. Now we identify $\bigl((M_E)\vert_C\bigr)_x \cong H^0(X,E\otimes m_x)$ and consider the following commutative diagram:
\begin{center}
 \begin{tikzcd}
  K_x \arrow[r," "] \arrow[d," "] & \bigl(F\vert_C\bigr)_x \arrow[d," "]\\
  H^0\bigl(X,E\otimes\mathcal{O}_X(-H)\bigr) \arrow[r, " "] & \bigl((M_E)\vert_C)_x = H^0(X,E\otimes m_x) 
 \end{tikzcd}
\end{center}
Let $Z=\nu^{-1}\mathbb{P}_{sub}\bigl((F\vert_C)_x\bigr)$. Since $\nu$ is finite, $ \dim Z \leq \dim\bigl(\mathbb{P}_{sub}(F\vert_C)_x\bigr)$.

Given $s\in H^0\bigl(X,\mathcal{O}_X(H)\otimes m_x\bigr)$, we have that $s$ induces the injective morphism

$$H^0\bigl(X,E\otimes \mathcal{O}_X(-H)\bigr) \longrightarrow  H^0\bigl(X,E\otimes m_x\bigr) = (M_E)_x$$
$$\phi \longmapsto s\otimes \phi$$
Therefore for any $\phi$ in the image of the morphism $$K_x\longrightarrow H^0\bigl(X,E\otimes \mathcal{O}_X(-H)\bigr),$$ we have that $(s,\phi)\in Z$ and $\pi_1(s,\phi) = s$, where $\pi_1$ denotes the first projection onto\\ $\mathbb{P}_{sub}\bigl(H^0(X,\mathcal{O}_X(H)\otimes m_x)\bigr)$. It follows that the projection $$\pi_1 : Z \longrightarrow \mathbb{P}_{sub}\bigl(X,H^0\bigl(\mathcal{O}_X(H)\otimes m_x\bigr)\bigr)$$ is dominant, and the dimension of the general fiber is greater than or equal to $\rk(K)-1$.

Also, $h^0\bigl(\mathcal{O}_X(H)\otimes m_x\bigr) = h^0\bigl(\mathcal{O}_X(H)\bigr)-1$.

Hence $$\rk(F)\geq h^0\bigl(X, \mathcal{O}_X(H)\bigr)+\rk(K)-1.$$

As $K$ is a subsheaf of the trivial bundle $\mathcal{O}_C^{\oplus h^0\bigl(X,E\otimes \mathcal{O}_X(-H)\bigr)}$, we have $\deg(K) < 0$.

Therefore we observe that using short exact sequence
\begin{align}\label{seq4}
 0\longrightarrow K \longrightarrow F\vert_C \longrightarrow N \longrightarrow 0
\end{align}

\begin{center}
 $$\mu(F\vert_C) =  \frac{\deg(K)+\deg(N)}{\rk(F)} \leq \frac{\deg(N)}{\rk(F)}\leq \frac{\deg(N)}{\rk(N)}\frac{\rk(N)}{\rk(F)} = \mu(N)\frac{\rk(N)}{\rk(F)}$$
\end{center}
Since $M_{E\vert_C}$ is semistable and $N\hookrightarrow M_{E\vert_C}$, we get $\mu(N)\leq \mu\bigl(M_{E\vert_C}\bigr)$.
Thus
$$\mu(F\vert_C) \leq \mu(M_{E\vert_C})\Bigl(\frac{\rk(F)-\rk(K)}{\rk(F)}\Bigr) = \mu(M_{E\vert_C}\bigr)\Bigl(1-\frac{\rk(K)}{\rk(F)}\Bigr).$$

Recall that from  (\ref{seq3}) we have $$\deg(M_{E\vert_C}) = \deg\bigl((M_E)\vert_C\bigr),\hspace{1mm} \text{and} \hspace{1mm} \rk\bigl((M_E)\vert_C\bigr) = \rk\bigl(M_{E\vert_C}\bigr)+h^0\bigl(X, E\otimes \mathcal{O}_X(-H)\bigr).$$

Also  we have $\mu\bigl((M_E)\vert_C\bigr)\leq \mu(F\vert_C)$ from ($\ref{seq1}$).
This implies $$\mu\bigl((M_E)\vert_C\bigr) \leq \mu(M_{E\vert_C}\bigr)\Bigl(1-\frac{\rk(K)}{\rk(F)}\Bigr),$$
i.e. $$\frac{\deg\bigl((M_E)\vert_C\bigr)}{\rk\bigl(M_{E\vert_C}\bigr)+h^0\bigl(X, E\otimes \mathcal{O}_X(-H)\bigr)}<\frac{\deg\bigl((M_E)\vert_C)}{\rk\bigl(M_{E\vert_C}\bigr)}\Bigl(1-\frac{\rk(K)}{\rk(F)}\Bigr).$$
\vspace{2mm}

As $E\vert_C$ is ample, we conclude $ \deg(M_{E\vert_C})=\deg\bigl((M_E)\vert_C\bigr) < 0$. Hence  we  have
$$\frac{1}{\rk\bigl(M_{E\vert_C}\bigr)+h^0\bigl(X,E\otimes \mathcal{O}_X(-H)\bigr)}>\frac{1}{\rk\bigl(M_{E\vert_C}\bigr)}\Bigl(1-\frac{\rk(K)}{\rk(F)}\Bigr),$$
i.e. $$\frac{\rk\bigl(M_{E\vert_C}\bigr)}{\rk\bigl(M_{E\vert_C}\bigr)+h^0\bigl(X,E\otimes \mathcal{O}_X(-H)\bigr)}>\Bigl(1-\frac{\rk(K)}{\rk(F)}\Bigr).$$
Thus $$\frac{\rk(K)}{\rk(F)} > 1- \frac{\rk\bigl(M_{E\vert_C}\bigr)}{\rk\bigl(M_{E\vert_C}\bigr)+h^0\bigl(X,E\otimes \mathcal{O}_X(-H)\bigr)} = \frac{h^0\bigl(X,E\otimes\mathcal{O}_X(-H)\bigr)}{\rk\bigl((M_E)\vert_C\bigr)}.$$
i.e. $$\rk(K) > h^0\bigl(X,E\otimes\mathcal{O}_X(-H)\bigr) \frac{\rk(F)}{\rk\bigl((M_E)\vert_C\bigr)}.$$
Recall that $\rk(F) \geq h^0\bigl(\mathcal{O}_X(H)\bigr)+\rk(K)-1$. Hence we conclude that $$\rk(F) > h^0(\mathcal{O}_X(H))+ h^0\bigl(X,E\otimes\mathcal{O}_X(-H)\bigr)\Bigl(\frac{\rk(F)}{\rk\bigl((M_E)\vert_C\bigr)}\Bigr)-1.$$
i.e., $$\rk(F)-h^0\bigl(X,E\otimes\mathcal{O}_X(-H)\bigr)\Bigl(\frac{\rk(F)}{\rk\bigl((M_E)\vert_C\bigr)}\Bigr)> h^0(X,\mathcal{O}_X(H))-1.$$
Equivalently, $$\rk(F)\Bigl(\frac{\rk(M_E)-h^0\bigl(X,E\otimes\mathcal{O}_X(-H)\bigr)}{\rk(M_E)}\Bigr) > h^0(X,\mathcal{O}_X(H))-1.$$
\vspace{1mm}

Recall using short exact sequence (\ref{seq2}) we have  $$\rk(M_E) = \rk\bigl((M_E)\vert_C\bigr) = h^0\bigl(X,E\otimes\mathcal{O}_X(-H)\bigr) + \rk(M_{E\vert_C}).$$
Now from the exact sequence $0\longrightarrow M_{E\vert_C}\longrightarrow \mathcal{O}_C^{\oplus h^0(C,E\vert_C)}\longrightarrow E\vert_C\longrightarrow 0$, we have  $$\rk\bigl(M_{E\vert_C}\bigr) = h^0(C,E\vert_C)-\rk(E) > 0.$$
This implies $$\rk(M_E) - h^0(X,E\otimes\mathcal{O}_X(-H)) = \rk\bigl(M_{E\vert_C}\bigr) = h^0(C,E\vert_C)-\rk(E) > 0.$$
Hence we get $$\rk(F) > \frac{(h^0(X,\mathcal{O}_X(H))-1)\rk(M_E)}{\rk(M_E)-h^0\bigl(X,E\otimes\mathcal{O}_X(-H)\bigr)} = \frac{(h^0(X,\mathcal{O}_X(H))-1)\rk(M_E)}{h^0(C,E\vert_C)-\rk(E)} \geq \rk(M_E),$$
which is a contradiction (see hypothesis (3)) as $F$ is a subbundle of $M_E$. This completes the proof.
\end{proof}
\begin{thm}\label{thm3.4}
  Let $X$ be a smooth complex surface and we fix a very ample line bundle $H$ on $X$. Let $E$ be an $H$-semistable globally generated vector bundle of rank $r$ on $X$. Then for large enough $m\gg 0$, the associated sygyzy bundle $M_{E(m)}$ is $H$-stable.
  \begin{proof}
Let $m\gg 0,n\gg 0 $ be large enough integers such that  $(m-n)\gg0$ and satisfies the followings
\begin{itemize}
 \item $E(m)$ is ample,
 \item $E(m)\otimes \mathcal{O}_X(-(m-n)H) = E(n)$ is globally generated, and $H^1(X,E(n))=0$.
 \item  $nH-K_X$ is very ample,
\end{itemize}

 As $E$ is $H$-semistable, $E(m)$ is also $(m-n)H$-semistable. Fix a smooth curve $C\in \lvert (m-n)H \rvert$ such that $E(m)\vert_C$ is semistable. This can be done using Mehta-Ramanathan restriction theorem (see \cite[Theorem 7.2.1]{HL10}). Further assume that $H^i\bigl(X,\mathcal{O}_X((m-n)H) = 0$ for $i=1,2$ and $H^1(C,E(m)\vert_C)=0$ for large enough $(m-n)\gg 0$ and $m\gg 0$.
\vspace{1mm}

We will show that $M_{E(m)}$ is $(m-n)H$-stable, and hence $M_{E(m)}$ is $H$-stable. To prove this, we will show that the smooth curve $C\in \lvert (m-n)H \rvert$  satisfies the hypothesis of Theorem \ref{thm1}.

(i) Note that $$H^1\bigl(X, E(m)\otimes \mathcal{O}_X(-(m-n)H)\bigr) = H^1(X, E(n)) = 0.$$

(ii) Next we will show that if $W \subseteq H^0(X,E(n))$ is a subspace which generates $E(n)$ then the natural multiplication map
\begin{center}
$H^0\bigl(X,\mathcal{O}_X(m-n)H\otimes m_x\bigr)\otimes W\longrightarrow H^0(X,E(m)\otimes m_x)$
\end{center}
is surjective. It is enough to prove that the map
$$H^0\bigl(X,\mathcal{O}_X(m-n)H\bigr)\otimes W \longrightarrow H^0(X,E(m))$$ is surjective for $m\gg 0$. As $W$ generates $E(n)$, we have a surjection $$W\otimes \mathcal{O}_X(-n)\longrightarrow E \longrightarrow 0.$$
Let $K$ be the kernel of this surjetive map, and then we have the following short exact sequence :
\begin{align}\label{seq6}
0\longrightarrow K \longrightarrow W\otimes \mathcal{O}_X(-n)\longrightarrow E \longrightarrow 0.
\end{align}
Let $m\gg 0$ be large enough so that $H^1(X,K(m))=0$. Tensoring (\ref{seq6}) with $\mathcal{O}_X(m)$ we get
\begin{align}\label{seq5}
0\longrightarrow K \otimes \mathcal{O}_X(m) \longrightarrow W\otimes \mathcal{O}_X(m-n)\longrightarrow E\otimes \mathcal{O}_X(m) \longrightarrow 0.
\end{align}
Passing onto long exact sequence we have
the required surjection
$$H^0\bigl(X,\mathcal{O}_X(m-n)H\bigr)\otimes W \longrightarrow H^0(X,E(m))\longrightarrow 0.$$

(iii) Next we will show that
$$\dim_{\mathbb{C}}H^0\bigl(X,\mathcal{O}_X((m-n)H)\bigr) \geq \dim_{\mathbb{C}}H^0\bigl(C,E(m)\vert_C)-r+1.$$

Note that by Riemann-Roch theorem for line bundles on a smooth surface, we have  $$h^0\bigl(X,\mathcal{O}_X((m-n)H)\bigr) = \frac{(m-n)H\cdot \bigl((m-n)H-K_X\bigr)}{2}+\chi(\mathcal{O}_X)$$
Also by Riemann-Roch Theorem for vector bundles on smooth curves, $$h^0(C, E(m)\vert_C) = \deg(E(m)\vert_C)+r(1-g).$$
By adjunction formula
\begin{align*}
g(C)=\frac{K_X\cdot C+ C^2}{2}+1
=\frac{(m-n)K_X\cdot H+(m-n)^2H^2}{2}+1
\end{align*}
Thus we have
\begin{align*}
&h^0(C,E(m)\vert_C) = \deg\bigl(E(m)\vert_C\bigr) +r(1-g(C))\\
&= \deg\bigl(E\vert_C\bigr)+r\deg\bigl(\mathcal{O}_C(mH\vert_C)\bigr)+r(1-g(C))\\
& = \deg(E\vert_C)+ r\deg\bigl(\mathcal{O}_C(mH\vert_C)\bigr) + r\Bigl\{1-\frac{(m-n)K_X\cdot H+(m-n)^2H^2}{2}-1\Bigr\}\\
&=c_1(E)\cdot (m-n)H+rm(m-n)H^2-r\Bigl\{\frac{(m-n)K_X\cdot H+(m-n)^2H^2}{2}\Bigr\}
\end{align*}
Hence
\begin{align*}
& h^0\bigl(X,\mathcal{O}_X((m-n)H)\bigr) -h^0\bigl(C,E(m)\vert_C)+r-1\\
&= \frac{(m-n)H\bigl((m-n)H-K_X\bigr)}{2}+\chi(\mathcal{O}_X) +r-1  \\
& \hspace{2cm} - c_1(E)\cdot (m-n)H - rm(m-n)H^2 +
 r\Bigl\{\frac{(m-n)K_X\cdot H+(m-n)^2H^2}{2}\Bigr\}\\
&=\frac{(m-n)}{2}\Bigl\{(m-n)H^2-2c_1(E)\cdot H -K_X\cdot H-2rmH^2 + rK_X\cdot H+r(m-n)H^2\Bigr\}\\
&\hspace{2cm}+\chi(\mathcal{O}_X)+r-1\\
&=\frac{(m-n)}{2}\Bigl\{r\bigl\{(m-n)H^2+K_X\cdot H-2mH^2\bigr\}+(m-n)H^2+K_X\cdot H +2c_1(E)\cdot H\Bigr\} \\
&\hspace{2cm}+\chi(\mathcal{O}_X)+r-1\\
& \geq 0
\end{align*}
for $(m-n) \gg 0$.

Hence we conclude for large enough $(m-n)\gg 0$ and $m\gg 0$, we have
$$h^0\bigl(X,\mathcal{O}_X((m-n)H)\bigr) \geq h^0\bigl(C,E(m)\vert_C)+r-1.$$

(iv) Next we will show that  $M_{E(m)\vert_C}$ is a semistable bundle on $C$.  Precisely, we will show that  $$\mu\bigl(E(m)\vert_C\bigr) \geq 2g(C).$$ Hence by Butler's Theorem \cite[Theorem 1.2]{B94}, $M_{E(m)\vert_C}$ is semistable.
\vspace{1mm}

  Write $mH = K_X+ mH -nH + P$ where $P = nH-K_X$ is very ample.

  By adjunction formula we get
  $$ K_C = (K_X+C)\vert_C.$$

  This implies $$\deg(K_C) = \deg\bigl(\mathcal{O}_X(mH)\vert_C\bigr) - P\cdot C$$

  Now $$\deg(K_C) = \deg\bigl((K_X+C)\vert_C\bigr) = 2g(C)-2$$

  Therefore,
  $$\deg\bigl(\mathcal{O}_X(mH)\bigr) = 2g(C) -2 +P\cdot C$$

  As $P$ is very ample and $C\sim (m-n)H$, we can assume that $P\cdot C \geq 2$.

  Hence we conclude that $$\deg\bigl(\mathcal{O}_X(mH)\bigr) \geq 2g(C).$$

  Note that
  \begin{align*}
  &\deg\bigl(E(m)\vert_C\bigr)\\
  &= c_1\bigl(E(m)\bigr)\cdot C\\
  &= c_1\bigl(E\otimes \mathcal{O}_X(mH)\bigr)\cdot C\\
  &= \bigl(c_1(E)+\rk(E)\cdot c_1(mH)\bigr) \cdot C \\
  &= \deg(E\vert_C) + \rk(E)\deg\bigl(\mathcal{O}_X(mH)\vert_C\bigr) \geq \deg(E\vert_C)+\rk(E)\bigl(2g(C)).
  \end{align*}
As $E$ is globally generated, it's restriction $E\vert_C$ is also globally generated. Hence $\deg(E\vert_C) \geq 0$.

  Thus $$\mu\bigl(E(m)\vert_C\bigr)\geq 2g(C)$$ as required. This completes the proof.
\end{proof}
 \end{thm}
Based on results of this article, we  ask the following question.
\begin{qn}
 Let $X$ be an irreducible smooth complex projective variety and $H$ be a very ample line bundle  on $X$. Let $E$ be an $H$-semistable vector bundle of rank $r$ on $X$. Then for large enough $m\gg 0$, is the associated sygyzy bundle $M_{E(m)}$ slope $H$-stable?
\end{qn}

\end{document}